\documentclass[reqno,11pt]{amsart}
\usepackage{amsmath,amssymb,latexsym,cancel,rotating}

\numberwithin{equation}{section}

\newtheorem{theorem}{Theorem}[section]
\newtheorem{proposition}[theorem]{Proposition}

\newtheorem{lemma}[theorem]{Lemma}

\newtheorem{remark}[theorem]{Remark}

\newtheorem{definition}[theorem]{Definition}

\input xy
\xyoption{poly}
\xyoption{2cell}
\xyoption{all}

\def\A{\mathcal{A}}

\def\ZZ{\mathbb Z}

\def\de{\delta}

\def\ZZ{\mathbb{Z}}

\renewcommand{\eqref}[1]{{\rm (\ref{#1})}}

\begin{document}

\title[A generalized quantum cluster algebra of Kronecker type]
{A generalized quantum cluster algebra of Kronecker type}

\author{Liqian Bai, Xueqing Chen, Ming Ding and Fan Xu}
\address{Department of Applied Mathematics, Northwestern Polytechnical University, Xi'an, Shaanxi 710072, P.R. China}
\email{bailiqian@nwpu.edu.cn (L.Bai)}
\address{Department of Mathematics,
 University of Wisconsin-Whitewater\\
800 W. Main Street, Whitewater, WI.53190. USA}
\email{chenx@uww.edu (X.Chen)}
\address{School of Mathematics and Information Science\\
Guangzhou University, Guangzhou 510006, P.R.China}
\email{dingming@gzhu.edu.cn (M.Ding)}
\address{Department of Mathematical Sciences\\
Tsinghua University\\
Beijing 100084, P.~R.~China} \email{fanxu@mail.tsinghua.edu.cn(F.Xu)}






\maketitle

\begin{abstract}
The cluster multiplication formulas  for a generalized quantum cluster algebra of Kronecker type are explicitly given. Furthermore,  a positive bar-invariant $\mathbb{Z}[q^{\pm\frac{1}{2}}]$-basis of this algebra is constructed.
\end{abstract}


\section{Introduction}

Cluster algebras were introduced by Fomin and Zelevinsky \cite{ca1,ca2} in order
to set up an algebraic framework for studying the total positivity and Lusztig's canonical bases.  Quantum cluster algebras, as the quantum deformations of cluster algebras,
were later introduced by Berenstein and Zelevinsky \cite{BZ2005} for studying the dual canonical bases in coordinate rings and their $q$-deformations.
An important feature of (quantum) cluster algebras is the so--called  Laurent phenomenon which says that all cluster variables belong to an intersection of certain (may be infinitely many) rings of Laurent polynomials.

Generalized cluster algebras were introduced by Chekhov and Shapiro \cite{CS} in order to understand the Teichm\"uller theory of hyperbolic orbifold surfaces. The exchange relations for cluster variables of generalized cluster algebras are polynomial exchange relations while the exchange relations for cluster algebras are binomial relations. Generalized cluster algebras also possess the Laurent phenomenon \cite{CS}  and  are studied by many people in a similar way as cluster algebras (see for example \cite{nak2, CL1, CL2, Mou, BCDX-1}).
As a natural generalization of both quantum cluster algebras and generalized cluster algebras,  we defined the generalized quantum cluster algebras \cite{BCDX}. Not surprised that the Laurent phenomenon also holds in  these algebras \cite{BCDX-2}.

One of the most important problems in cluster theory is to  construct cluster multiplication formulas. For acyclic cluster algebras, Sherman and Zelevinsky \cite{SZ} firstly established the cluster multiplication formulas in rank 2 cluster algebras of finite and affine types. Cerulli \cite{CI-1} generalized this result to rank $3$ cluster algebra of affine type $A_{2}^{(1)}$.  Caldero and Keller \cite{ck}  constructed the cluster multiplication formulas  between two cluster characters for simply laced Dynkin quivers, which was generalized to affine types by Hubery in~\cite{Hubery} and to acyclic types by Xiao and Xu in~\cite{XX,X}.  In the quantum case, Ding and Xu \cite{dx} firstly gave the cluster multiplication formulas of the quantum  cluster algebra of Kronecker type.
Recently, Chen, Ding and Zhang \cite{cdz} obtained the cluster multiplication formulas  in the acyclic quantum cluster algebras with arbitrary coefficients through some quotients of derived Hall algebras of acyclic valued quivers.
One of the most powerful tools in cluster theory is the cluster multiplication formulas which are very useful to construct bases of (quantum) cluster algebras with nice properties (see for example \cite{SZ, ck, CI-1, dx, dxc}).  In cluster theory, a basis is called positive if its structure constants are positive.
Several positive bases such as the atomic bases and the triangular bases of some (quantum) cluster algebras have been found (see \cite{qin2, li-pan}).
So far, no similar results have been obtained in generalized  quantum cluster algebras. It becomes natural to think whether one can give an explicit treatment of  the above mentioned problems for  generalized quantum cluster algebras.

In this paper, we study a generalized quantum cluster algebra of Kronecker type denoted by $\A_{q}(2,2)$.
We recall the definition of  generalized quantum cluster algebras in Section 2, provide the cluster multiplication formulas of $\A_{q}(2,2)$ in Section 3, and explicitly construct a positive bar-invariant $\mathbb{Z}[q^{\pm\frac{1}{2}}]$-basis of $\A_{q}(2,2)$ in Section 4.

\section{Preliminaries}
In this section, we mainly review  the definition of generalized quantum cluster algebras~\cite{BCDX}. Throughout this section, $m$ and $n$ are positive integers with $m\geq n$. Let $\widetilde{B}=(b_{ij})$ be an $m\times n$ integer matrix whose upper $n\times n$ submatrix is denoted by $B$ and $\Lambda=(\lambda_{ij})$ a skew-symmetric $m\times m$ integer matrix.
\begin{definition}
The pair $(\Lambda,\widetilde{B})$ is called compatible if for any $1\leq i\leq m$ and $1\leq j\leq n$, we have

\begin{equation}\label{equ1}
\sum^{m}_{k=1}\lambda_{ki}b_{kj}=\left\{
  \begin{aligned}
    d_{j}&,~\text{if}~i=j;  \\
    0&,~\text{otherwise};
  \end{aligned}
  \right.
\end{equation}
for some positive integers $\widetilde{d}_j\ (1\leq j\leq n)$.
\end{definition}

Note that the skew-symmetric matrix $\Lambda$  gives a skew-symmetric bilinear form on $\ZZ^{m}$ defined by $$\Lambda(\mathbf{a},\mathbf{b})=\mathbf{a}^T\Lambda \mathbf{b}$$ for any column vectors $\mathbf{a},\mathbf{b}\in\ZZ^{m}$.

Let $q$ be a formal variable and $\mathbb{Z}[q^{\pm\frac{1}{2}}]\subset \mathbb{Q}(q^{\frac{1}{2}})$  the ring of integer Laurent polynomials in $q^{\frac{1}{2}}$.
One can associate to $(\Lambda, q)$ a quantum torus algebra as follows.
\begin{definition}
The quantum torus $\mathcal{T}$ over $\ZZ[q^{\pm\frac{1}{2}}]$ is generated by the symbols $\{X(\mathbf{a})~|~\mathbf{a}\in\ZZ^{m}\}$ subject to the multiplication relations
\begin{equation}\label{qcommutation}
X(\mathbf{a})X(\mathbf{b})=q^{\frac{1}{2}\Lambda(\mathbf{a},\mathbf{b})}X(\mathbf{a}+\mathbf{b}),
\end{equation}
for any $\mathbf{a},\mathbf{b}\in\ZZ^{m}$.
\end{definition}
The skew-field of fractions of $\mathcal{T}$ is denoted by $\mathcal{F}$. On the quantum torus $\mathcal{T}$, the $\mathbb{Z}$-linear bar-involution is defined by setting $$\overline{q^{\frac{r}{2}}X(\mathbf{a})}=q^{-\frac{r}{2}}X(\mathbf{a})$$ for any $r\in \ZZ$ and $\mathbf{a}\in\ZZ^{m}$.

Let  $e_k$ be the $k$-th standard unit vector in $\ZZ^{m}$ and set $X_k=X(e_k)$  for $1\leq k\leq m$. An easy computation shows  that
$$
X(\mathbf{a})=q^{\frac{1}{2}\sum\limits_{i< j}\lambda_{ji}a_ia_j}X_{1}^{a_1}X_{2}^{a_2}\ldots X_{m}^{a_m}
$$
for $\mathbf{a}=(a_1,a_2,\ldots,a_m)^{T}\in\ZZ^{m}$.

For any $1\leq i\leq n$, we say that $\widetilde{B}'=(b'_{kl})$ is obtained from
the matrix $\widetilde{B}=(b_{kl})$ by the matrix mutation in the direction $i$ if $\widetilde{B}':=\mu_i\big(\widetilde{B}\big)$ is given by
\begin{gather*}
b'_{kl}=
 \begin{cases}
 -b_{kl}, &\text{if}\quad k=i\quad \text{or}\quad l=i,
 \\
b_{kl}+\frac{|b_{ki}|b_{il}+b_{ki}|b_{il}|}{2}, &\text{otherwise}.
 \end{cases}
\end{gather*}

Denote the function
\begin{gather*}
[x]_+=
 \begin{cases}
 x, &\text{if}\quad x\geq 0;
 \\
0, & \text{if}\quad x\leq 0.
 \end{cases}
\end{gather*}

For any $1\leq i\leq n$ and a sign $\varepsilon\in\{\pm1\}$,  denote by $E_{\varepsilon}$  the $m\times m$ matrix associated to the matrix $\widetilde{B}=(b_{ij})$ with entries as follows
\begin{align*}
(E_{\varepsilon})_{kl}=\left\{
  \begin{aligned}
  \delta_{kl},{\hskip 0.7cm}&~\text{if}~l\neq i;\\
  -1,{\hskip 0.7cm}&~\text{if}~k=l=i;\\
  [-\varepsilon b_{ki}]_+,&~\text{if}~k\neq l=i.
  \end{aligned}
\right.
\end{align*}

\begin{proposition}\cite[Proposition 3.4]{BZ2005}\label{com}
Let $(\Lambda,\widetilde{B})$ be a compatible pair, then the pair $(\Lambda^{\prime},{\widetilde{B}}^{\prime})$ is also compatible and independent of the choice of
$\varepsilon$, where $\Lambda^{\prime}=E_{\varepsilon}^T\Lambda E_{\varepsilon}$ and $\widetilde{B}^\prime=\mu_i(\widetilde{B})$.
\end{proposition}
We say that the compatible pair $(\Lambda^{\prime},{\widetilde{B}}^{\prime})$ is obtained from the compatible pair $(\Lambda, {\widetilde{B}})$ by mutation in the direction $i$ and denoted by $\mu_i(\Lambda, {\widetilde{B}})$.  It is  known that $\mu_i$ is an involution \cite[Proposition 3.6]{BZ2005}.

For each $1\leq i\leq n$, let $d_i$ be a positive integer such that $\frac{b_{li}}{d_i}$ are integers for all $1\leq l\leq m$ and denote by $\beta^{i}=\frac{1}{d_i}\mathbf{b}^{i}$, where $\mathbf{b}^{i}$ is the $i$-th column of $\widetilde{B}$. Denote by
$$
\mathbf{h}_i=\{h_{i,0}(q^{\frac{1}{2}}),h_{i,1}(q^{\frac{1}{2}}),\ldots,h_{i,d_i}(q^{\frac{1}{2}})\},  \,\,  1\leq i\leq n,$$
where $h_{k,l}(q^{\frac{1}{2}})\in\ZZ[q^{\pm\frac{1}{2}}]$ satisfying that $h_{k,l}(q^{\frac{1}{2}})= h_{k,d_k-l}(q^{\frac{1}{2}})$ and $h_{k,0}(q^{\frac{1}{2}})= h_{k,d_k}(q^{\frac{1}{2}})=1$. We set $\mathbf{h}:=(\mathbf{h}_1,\mathbf{h}_2,\ldots,\mathbf{h}_n)$.

\begin{definition}
With the above notations,  the quadruple $(X,\mathbf{h},\Lambda,\widetilde{B})$ is called a quantum seed if the pair $(\Lambda,\widetilde{B})$ be compatible. For a given quantum seed $(X,\mathbf{h},\Lambda,\widetilde{B})$ and each $1\leq i \leq n$, the new quadruple $$(X^\prime,\mathbf{h}^\prime,\Lambda^\prime,\widetilde{B}^\prime):=\mu_i(X,\mathbf{h},\Lambda,\widetilde{B})$$ is defined by
\begin{equation}\label{clustervarialbemutaion}
X^\prime(e_k)=\mu_i(X(e_k))=\left\{
  \begin{aligned}
    X(e_k), {\hskip 6.5cm}&~\text{if}~k\neq i;  \\
\sum\limits_{r=0}^{d_i}h_{i,r}(q^{\frac{1}{2}})X(r[\beta^i]_+ +(d_i-r)[-\beta^{i}]_+ -e_i), &~\text{if}~k=i,
  \end{aligned}
  \right.
\end{equation}
and
$$
\mathbf{h}^\prime=\mu_i(\mathbf{h})=\mathbf{h}~ \text{ and }~(\Lambda^\prime,\widetilde{B}^\prime)=\mu_i(\Lambda,\widetilde{B}).
$$
We say that the quadruple $\mu_i(X,\mathbf{h},\Lambda,\widetilde{B})$ is obtained from $(X,\mathbf{h},\Lambda,\widetilde{B})$ by  mutation in the direction $i$.
\end{definition}

\begin{proposition}\cite[Proposition 3.6]{BCDX}\label{seed}
Let the quadruple $(X,\mathbf{h},\Lambda,\widetilde{B})$ be a quantum seed, then the quadruple $\mu_i(X,\mathbf{h},\Lambda,\widetilde{B})$ is also a quantum seed.
\end{proposition}

Note that $\mu_i$ is an involution by \cite[Proposition 3.7]{BCDX}. Two quantum seeds are said to be mutation-equivalent if they can be obtained from
each other by a sequence of seed mutations. Given the initial quantum seed  $(X, \textbf{h}, \Lambda, \widetilde{B})$, let $(X^{\prime}, \textbf{h}^{\prime}, \Lambda^{\prime}, \widetilde{B}^{\prime})$ is mutation-equivalent to $(X, \textbf{h}, \Lambda, \widetilde{B})$. Denote by $X^{\prime}=\{X^{\prime}_1,\ldots,X^{\prime}_m\}$ which is called the extended cluster and the set $\{X^{\prime}_1,\ldots,X^{\prime}_n\}$ is called the cluster. The element $X^{\prime}_i$ is called a cluster variable for any  $1\leq i\leq n$ and $X^{\prime}_k$ a frozen variable for any $n+1\leq k\leq m$. Note that $X^{\prime}_k=X_k\ (n+1\leq k\leq m)$. For convenience, let $\mathbb{P}$ denote the multiplicative group generated by $X_{n+1},\ldots,X_{m}$ and $q^{\frac{1}{2}}$, and $\mathbb{ZP}$  the ring of the Laurent polynomials in $X_{n+1},\ldots, X_m$ with coefficients in $\ZZ[q^{\pm\frac{1}{2}}]$.
\begin{definition}\label{def of qgca}
Given the initial quantum seed $(X, \textbf{h}, \Lambda, \widetilde{B})$, the associated generalized quantum cluster algebra $\mathcal{A}(X, \textbf{h}, \Lambda, \widetilde{B})$  is the $\mathbb{ZP}$-subalgebra of $\mathcal{F}$ generated by all cluster variables from the quantum seeds which are mutation-equivalent to $(X, \textbf{h}, \Lambda, \widetilde{B})$.
\end{definition}

The following  Laurent phenomenon
is one of the most important results on generalized quantum cluster algebras.
\begin{theorem}\cite[Theorem 3.1]{BCDX-2}\label{laurent}
The generalized quantum cluster algebra $\mathcal{A}(X, \textbf{h}, \Lambda, \widetilde{B})$ is a
subalgebra of the ring of  Laurent polynomials in the cluster variables in any cluster over $\ZZ\mathbb{P}$.
\end{theorem}

\section{The cluster multiplication formulas of $\A_{q}(2,2)$}

In the following, we will consider  the generalized quantum cluster algebra associated with the initial seed $(X,\mathbf{h},\Lambda,B)$, where $\mathbf{d}=(2,2)$, $\textbf{h}_1=\textbf{h}_2=(1,h,1)$ with
$h\in \mathbb{Z}[q^{\pm\frac{1}{2}}]$ and $\overline{h}=h$,
\begin{equation*}
\Lambda=\left(
  \begin{array}{cc}
    0 & 1 \\
    -1 & 0 \\
  \end{array}
\right)
\text{~and~}
B=\left(
  \begin{array}{cc}
    0 & 2 \\
    -2 & 0 \\
  \end{array}
\right).
\end{equation*}
Note that $\Lambda^{T}B=\left(
  \begin{array}{cc}
    2 & 0 \\
    0 & 2 \\
  \end{array}
\right),$ and the based quantum torus is
$$\mathcal{T}=\mathbb{Z}[q^{\pm\frac{1}{2}}][X^{\pm 1}_1,X^{\pm 1}_2|X_1X_2=qX_2X_1].$$

The quiver   associated to the matrix $B$ is the Kronecker quiver $Q$:
$$
\xymatrix {1\bullet  \ar @<2pt>[r] \ar @<-2pt>[r]& \bullet 2}
$$
We call this algebra a generalized quantum cluster algebra of Kronecker type, denoted by $\A_{q}(2,2)$.
By the definition and the Laurent phenomenon, $\A_{q}(2,2)$ is the $\mathbb{Z}[q^{\pm\frac{1}{2}}]$-subalgebra of $\mathcal{T}$ generated by the cluster variables $\{X_k~|~k\in\mathbb{Z}\}$ which are obtained from the following exchange relations:
\begin{equation*}
  X_{k-1}X_{k+1}=qX^{2}_{k}+q^{\frac{1}{2}}hX_k+1.
\end{equation*}

Recall that the $n$-th Chebyshev polynomials of the first kind $F_{n}(x)$ is defined by
$$
F_0(x)=1,F_1(x)=x, F_2(x)=x^2-2, F_{n+1}(x)=F_{n}(x)x-F_{n-1}(x)~\text{for}~n\geq2,
$$
and $F_n(x)=0$ for $n<0$.

Denote
$$X_\delta:=q^{\frac{1}{2}}X_{0}X_3-q^{\frac{1}{2}}(q^{\frac{1}{2}}X_1+h)(q^{\frac{1}{2}}X_{2}+h),$$
thus $X_\delta\in \A_{q}(2,2)$.
\begin{lemma}\label{bar}
For each $n\in \mathbb{Z}_{>0}$, $F_{n}(X_\de)$ is a bar-invariant element in $\A_{q}(2,2)$.
\end{lemma}
\begin{proof}
An direct computation shows that
 $$X_\delta=X(-1,-1)+hX(-1,0)+hX(0,-1)+X(-1,1)+X(1,-1),$$
 thus $X_\de$ is a bar-invariant element in $\A_{q}(2,2)$. Then the proof  follows from the definition of the $n$-th Chebyshev polynomials  $F_{n}(x)$.
\end{proof}
We define an automorphism denoted by $\sigma$ on the generalized quantum cluster algebra $\A_{q}(2,2)$ as follows
$$\sigma(X_{k})=X_{k+1}\ \text{~and~} \ \sigma(q^{\frac{k}{2}})=q^{\frac{k}{2}},$$
for any $k\in \mathbb{Z}$. Then we have the following result which will be useful for us to prove the cluster multiplication formulas.
\begin{lemma}\label{auto}
For each $n\in \mathbb{Z}_{>0}$, $\sigma(F_{n}(X_\de))=F_{n}(X_\de).$
\end{lemma}
\begin{proof}
Note that
$$
\sigma(X_\de)=q^{\frac{1}{2}}X_1X_4-q^{\frac{1}{2}}(q^{\frac{1}{2}}X_2+h)(q^{\frac{1}{2}}X_3+h),
$$
$$
X_3=X(-1,2)+hX(-1,1)+X(-1,0)
$$
and
\begin{align*}
X_4=&X(-2,3)+(q^{-\frac{1}{2}}+q^{\frac{1}{2}})hX(-2,2)+(q^{-1}+h^2+q)X(-2,1) +(q^{-\frac{1}{2}} +q^{\frac{1}{2}})hX(-2,0) \\
&+X(-2,-1)+hX(-1,1)+h^2X(-1,0)+hX(-1,-1)+X(0,-1).
\end{align*}
Thus
\begin{align*}
q^{\frac{1}{2}}X_1X_4=&q^2X(-1,3)+(q+q^2)hX(-1,2)+(1+qh^2+q^2)X(-1,1)+(1+q)hX(-1,0) \\
&+X(-1,-1) +qhX(0,1)+q^{\frac{1}{2}}h^2+hX(0,-1)+X(1,-1)
\end{align*}
and
\begin{align*}
q^{\frac{1}{2}}(q^{\frac{1}{2}}X_2+h)(q^{\frac{1}{2}}X_3+h)=&q^2X(-1,3)+(q+q^2)hX(-1,2) +qhX(0,1)\\
&+(qh^2+q^2)X(-1,1)+qhX(-1,0)+q^{\frac{1}{2}}h^2.
\end{align*}
We obtain that
$$\sigma(X_\de)=X(-1,1)+hX(-1,0)+X(-1,-1)+hX(0,-1)+X(1,-1)=X_\de.$$
Then the proof  follows from the induction on $n$ and  the definition of the $n$-th Chebyshev polynomials  $F_{n}(x)$.
\end{proof}

For a real number $x$, denote the floor function  by  $\lfloor x\rfloor$ and the ceiling function by $\lceil x\rceil$.
The following Theorem \ref{multi-1} and Remark \ref{multi-2} give the explicit cluster multiplication formulas for  $\A_{q}(2,2)$.
\begin{theorem}\label{multi-1}
Let $m$ and $n$ be integers.
\begin{itemize}
\item[(1)] For any  $m> n\geq 1$, we have
\begin{align}\label{equation1}
F_{m}(X_\delta)F_{n}(X_\delta)=F_{m+n}(X_\delta)+F_{m-n}(X_\delta),\ F_{n}(X_\delta)F_{n}(X_\delta)=F_{2n}(X_\delta)+2.
\end{align}

\item[(2)] For any  $n\geq 1$, we have
\begin{align}\label{equation1}
X_mF_{n}(X_\delta)=q^{-\frac{n}{2}}X_{m-n}+q^{\frac{n}{2}}X_{m+n} +\sum\limits_{k=1}^{n}(\sum\limits_{l=1}^{k}q^{-\frac{k+1}{2}+l})hF_{n-k}(X_\delta).
\end{align}

\item[(3)] For any $n\geq 2$, we have
\begin{align}\label{equation2}
X_mX_{m+n}=&q^{\lfloor\frac{n}{2}\rfloor}X_{\lfloor m+\frac{n}{2}\rfloor}X_{\lceil m+\frac{n}{2}\rceil} +\sum\limits_{k=1}^{n-1}(\sum\limits_{l=1}^{\text{min}(k,n-k)}q^{-\frac{1}{2}+l})hX_{m+n-k}\nonumber \\ &+\sum\limits_{l=1}^{n-1}q^{-\frac{n-1-l}{2}}c_lF_{n-1-l}(X_\delta),
\end{align}
where $c_1=1$, $c_2=h^2$ and for $k\geq2$,
$$
c_{2k}=[\sum\limits_{i=1}^{k-1}a_i(q^{-(k-i)}+q^{k-i})+a_k]h^2
$$
and
\begin{align*}
c_{2k-1}=2[\sum\limits_{i=1}^{k-1}b_i(q^{-(k-i)}+q^{k-i})+b_k]h^2+\left\{
  \begin{aligned}
    \sum\limits_{i=1}^{\frac{k}{2}}(q^{-(k+1-2i)}+q^{k+1-2i}),{\hskip 0.5cm}~& ~\text{ if }~k~\text{is even};  \\
    \sum\limits_{i=1}^{\frac{k-1}{2}}(q^{-(k+1-2i)}+q^{k+1-2i})+1,& ~\text{ if }~k~\text{is odd},
  \end{aligned}
  \right.
\end{align*}
with $a_{i}=\frac{i(i+1)}{2}$ and
$$
b_i=\left\{
\begin{aligned}
\frac{i^2-1}{4},~ & ~\text{ if }~i~\text{is odd}; \\
\frac{i^2}{4},{\hskip 0.4cm}~ & ~\text{ if }~i~\text{is even}.
\end{aligned}
\right.
$$
\end{itemize}
\end{theorem}
\begin{proof}
(1)  The proof is immediately from the definition of the $n$-th Chebyshev polynomials  $F_{n}(x)$.

(2) By using the automorphism $\sigma$ repeatedly, it suffices to prove the following equation
$$
X_1F_n(X_\delta)=q^{-\frac{n}{2}}X_{1-n}+q^{\frac{n}{2}}X_{1+n} +\sum\limits_{k=1}^{n}(\sum\limits_{l=1}^{k}q^{-\frac{k+1}{2}+l})hF_{n-k}(X_\delta),
$$
for $n\geq0$.

When $n=1$,
\begin{align*}
X_1X_\delta=&X(1,0)(X(-1,-1)+hX(-1,0)+hX(0,-1)+X(-1,1)+X(1,-1))\\
=&q^{-\frac{1}{2}}X(0,-1)+h+q^{-\frac{1}{2}}hX(1,-1)+q^{\frac{1}{2}}X(0,1)+q^{-\frac{1}{2}}X(2,-1).
\end{align*}
Note that $X_0=X(2,-1)+hX(1,-1)+X(0,-1)$. Thus $X_1X_\delta=q^{-\frac{1}{2}}X_{0}+q^{\frac{1}{2}}X_2+h$. It follows that $$X_mX_{\delta}=q^{-\frac{1}{2}}X_{m-1}+q^{\frac{1}{2}}X_{m+1}+h$$ for all $m\in\ZZ$.

When $n=2$,
\begin{align*}
X_1F_2(X_\delta)=&X_1(X_{\delta}^{2}-2)=q^{-\frac{1}{2}}X_0X_\delta+ q^{\frac{1}{2}}X_2X_\delta+hX_\delta-2X_1 \\
=&q^{-1}X_{-1}+qX_3+(q^{-\frac{1}{2}}+q^{\frac{1}{2}})h+hX_\de.
\end{align*}

When $n\geq3$, assume that $X_1F_t(X_\de)=q^{-\frac{t}{2}}X_{1-t}+ q^{\frac{t}{2}}X_{1+t}+\sum\limits_{k=1}^{t}(\sum\limits_{l=1}^{k} q^{-\frac{k+1}{2}+l})hF_{n-k}(X_\de)$ for $t\leq n-1$.

If $t=n$, then
$$
X_1F_{n}(X_\de)=X_1(F_{n-1}(X_\de)X_\de-F_{n-2}(X_\de))=X_1F_{n-1}(X_\de)X_\de-X_1F_{n-2}(X_\de),
$$
\begin{align*}
&X_1F_{n-1}(X_\de)X_\de \\
=&q^{-\frac{n-1}{2}}X_{2-n}X_\de+q^{\frac{n-1}{2}}X_nX_\de+ \sum\limits_{k=1}^{n-1}(\sum\limits_{l=1}^{k} q^{-\frac{k+1}{2}+l})hF_{n-1-k}(X_\de)X_\de \\
=&q^{-\frac{n}{2}}X_{1-n}+q^{1-\frac{n}{2}}X_{3-n}+q^{\frac{n}{2}-1}X_{n-1}+q^{\frac{n}{2}}X_{n+1}+ (q^{-\frac{n-1}{2}}+q^{\frac{n-1}{2}})h\\
&+ \sum\limits_{k=1}^{n-1}(\sum\limits_{l=1}^{k} q^{-\frac{k+1}{2}+l})hF_{n-1-k}(X_\de)X_\de,
\end{align*}
and $X_1F_{n-2}(X_\de)=q^{1-\frac{n}{2}}X_{3-n}+q^{\frac{n}{2}-1}X_{n-1}+ \sum\limits_{k=1}^{n-2}(\sum\limits_{l=1}^{k} q^{-\frac{k+1}{2}+l})hF_{n-2-k}(X_\de)$. Note that
\begin{align*}
&\sum\limits_{k=1}^{n-1}(\sum\limits_{l=1}^{k} q^{-\frac{k+1}{2}+l})hF_{n-1-k}(X_\de)X_\de -\sum\limits_{k=1}^{n-2}(\sum\limits_{l=1}^{k} q^{-\frac{k+1}{2}+l})hF_{n-2-k}(X_\de) \\
=&\sum\limits_{k=1}^{n-3}(\sum\limits_{l=1}^{k} q^{-\frac{k+1}{2}+l})h(F_{n-1-k}(X_\de)X_\de-F_{n-2-k}(X_\de)) +(\sum\limits_{l=1}^{n-2} q^{-\frac{n-1}{2}+l})h(X_{\de}^{2}-2)\\
&+(\sum\limits_{l=1}^{n-2} q^{-\frac{n-1}{2}+l})h  +(\sum\limits_{l=1}^{n-1} q^{-\frac{n}{2}+l})hX_\de\\
=&\sum\limits_{k=1}^{n-1}(\sum\limits_{l=1}^{k} q^{-\frac{k+1}{2}+l})hF_{n-k}(X_\de)+ (\sum\limits_{l=1}^{n-2} q^{-\frac{n-1}{2}+l})h
\end{align*}
and $\sum\limits_{l=1}^{n-2}q^{-\frac{n-1}{2}+l}h+(q^{-\frac{n-1}{2}}+q^{\frac{n-1}{2}})h= \sum\limits_{l=1}^{n}q^{-\frac{n+1}{2}+l}h$.

It follows that $X_1F_{n}(X_\de)=q^{-\frac{n}{2}}X_{1-n}+q^{\frac{n}{2}}X_{n+1} +\sum\limits_{k=1}^{n}(\sum\limits_{l=1}^{k}q^{-\frac{k+1}{2}+l})hF_{n-k}(X_\de)$.

(3) In order to prove (\ref{equation2}), it suffices to show that
$$
X_1X_{1+n}=q^{\lfloor\frac{n}{2}\rfloor}X_{\lfloor 1+\frac{n}{2}\rfloor}X_{\lceil 1+\frac{n}{2}\rceil} +\sum\limits_{k=1}^{n-1}(\sum\limits_{l=1}^{\text{min}(k,n-k)}q^{-\frac{1}{2}+l})hX_{1+n-k} +\sum\limits_{l=1}^{n-1}q^{-\frac{n-1-l}{2}}c_lF_{n-1-l}(X_\delta)
$$
for $n\geq1$.

When $n=2$, it is the exchange relation. When $n=3$, by (\ref{equation1}), we have that
\begin{align*}
&X_1X_4\\
=&X_1(q^{-\frac{1}{2}}X_3X_\de-q^{-1}X_2-q^{-\frac{1}{2}}h)\\
=&q^{-\frac{1}{2}}(qX_{2}^{2}+q^{\frac{1}{2}}hX_2+1)X_\de -q^{-1}X_1X_2-q^{-\frac{1}{2}}hX_1\\
=&q^{\frac{1}{2}}X_2(q^{-\frac{1}{2}}X_1+q^{\frac{1}{2}}X_3+h)+h(q^{-\frac{1}{2}}X_1+q^{\frac{1}{2}}X_3+h) +q^{-\frac{1}{2}}X_\de -q^{-1}X_1X_2-q^{-\frac{1}{2}}hX_1\\
=&qX_2X_3+q^{\frac{1}{2}}hX_2+q^{\frac{1}{2}}hX_3+q^{-\frac{1}{2}}X_\de+h^2.
\end{align*}

Assume that
\begin{align*}
X_1X_{1+t}=q^{\lfloor \frac{t}{2}\rfloor}X_{\lfloor1+\frac{t}{2}\rfloor}X_{\lceil1+\frac{t}{2}\rceil}+ \sum\limits_{k=1}^{t-1}(\sum\limits_{l=1}^{\text{min}(k,t-k)}q^{-\frac{1}{2}+l})hX_{1+t-k} +\sum\limits_{l=1}^{t-1}q^{-\frac{t-1-l}{2}}c_lF_{t-1-l}(X_\de)
\end{align*}
for all $t\leq n-1$.

Note that $X_1X_{n+1}=q^{-\frac{1}{2}}X_1X_nX_\de-q^{-1}X_1X_{n-1}-q^{-\frac{1}{2}}hX_1$.

When $n$ is even and $n\geq 4$, then
$$
X_1X_n=q^{\frac{n}{2}-1}X_{\frac{n}{2}}X_{\frac{n}{2}+1}+\sum\limits_{k=1}^{n-2}(\sum\limits_{l=1}^{\text{min}(k,n-1-k)} q^{-\frac{1}{2}+l})hX_{n-k} +\sum\limits_{l=1}^{n-2}q^{-\frac{n-2-l}{2}}c_lF_{n-2-l}(X_\de),
$$
$$
q^{-1}X_1X_{n-1}=q^{\frac{n}{2}-2}X_{\frac{n}{2}}^{2}+ \sum\limits_{k=1}^{n-3}(\sum\limits_{l=1}^{\text{min}(k,n-2-k)} q^{-\frac{3}{2}+l})hX_{n-1-k} +\sum\limits_{l=1}^{n-3}q^{-\frac{n-1-l}{2}}c_lF_{n-3-l}(X_\de)
$$
and
\begin{align*}
&q^{-\frac{1}{2}}X_1X_nX_\de\\
=&q^{\frac{n-3}{2}}X_{\frac{n}{2}}X_{\frac{n}{2}+1}X_\de+ \sum\limits_{k=1}^{n-2}(\sum\limits_{l=1}^{\text{min}(k,n-1-k)}q^{-1+l}) hX_{n-k}X_\de +\sum\limits_{l=1}^{n-2}q^{-\frac{n-1-l}{2}}c_lF_{n-2-l}(X_\de)X_\de \\
=&q^{\frac{n}{2}-2}X_{\frac{n}{2}}^{2}+q^{\frac{n}{2}}X_{\frac{n}{2}+1}^{2}+q^{\frac{n-1}{2}}hX_{\frac{n}{2}+1}+q^{\frac{n}{2}-1} +q^{\frac{n-3}{2}}hX_{\frac{n}{2}} +\sum\limits_{k=1}^{n-2}(\sum\limits_{l=1}^{\text{min}(k,n-1-k)}q^{-\frac{3}{2}+l}) hX_{n-1-k}\\ &+\sum\limits_{k=1}^{n-2}(\sum\limits_{l=1}^{\text{min}(k,n-1-k)}q^{-\frac{1}{2}+l}) hX_{n+1-k} +\sum\limits_{k=1}^{n-2}(\sum\limits_{l=1}^{\text{min}(k,n-1-k)}q^{-1+l}) h^{2} \\
&+\sum\limits_{l=1}^{n-2}q^{-\frac{n-1-l}{2}} c_lF_{n-2-l}(X_\de)X_\de.
\end{align*}

Note that
\begin{align*}
\left\{
\begin{aligned}
k\leq n-2-k,&~\text{if~}1\leq k\leq \frac{n}{2}-1,\\
k>n-2-k,&~\text{if~}\frac{n}{2}\leq k\leq n-3,\\
k<n-1-k,&~\text{if~}1\leq k\leq \frac{n}{2}-1,\\
k>n-1-k,&~\text{if~}\frac{n}{2}\leq k\leq n-3,\\
k< n-k,{\hskip 0.6cm}&~\text{if~}1\leq k\leq \frac{n}{2}-1,\\
k\geq n-k,{\hskip 0.6cm}&~\text{if~}\frac{n}{2}\leq k\leq n-1.
\end{aligned}
\right.
\end{align*}

It follows that
\begin{align*}
&\sum\limits_{k=1}^{n-2}(\sum\limits_{l=1}^{\text{min}(k,n-1-k)}q^{-\frac{3}{2}+l}) hX_{n-1-k} -\sum\limits_{k=1}^{n-3}(\sum\limits_{l=1}^{\text{min}(k,n-2-k)}q^{-\frac{3}{2}+l}) hX_{n-1-k}-q^{-\frac{1}{2}}hX_1 \\
=&\sum\limits_{k=\frac{n}{2}}^{n-3}q^{-\frac{5}{2}+n-k}hX_{n-1-k}=\sum\limits_{k=\frac{n}{2}+2}^{n-1}q^{-\frac{1}{2}+n-k}hX_{n+1-k}.
\end{align*}

Hence
\begin{align*}
&\sum\limits_{k=1}^{n-1}(\sum\limits_{l=1}^{\text{min}(k,n-k)}q^{-\frac{1}{2}+l})hX_{n+1-k}\\
=&\sum\limits_{k=1}^{n-2}(\sum\limits_{l=1}^{\text{min}(k,n-1-k)}q^{-\frac{3}{2}+l})hX_{n-1-k} +\sum\limits_{k=1}^{n-2}(\sum\limits_{l=1}^{\text{min}(k,n-1-k)}q^{-\frac{1}{2}+l})hX_{n+1-k}+q^{\frac{n-1}{2}}hX_{\frac{n}{2}+1}\\ &+q^{\frac{n-3}{2}}hX_{\frac{n}{2}} -\sum\limits_{k=1}^{n-3}(\sum\limits_{l=1}^{\text{min}(k,n-2-k)}q^{-\frac{3}{2}+l})hX_{n-1-k} -q^{-\frac{1}{2}}hX_1.
\end{align*}

Because
\begin{align*}
&q^{\frac{n}{2}-1}+ \sum\limits_{k=1}^{n-2}(\sum\limits_{l=1}^{\text{min}(k,n-1-k)}q^{-1+l})h^{2} +\sum\limits_{l=1}^{n-2}q^{-\frac{n-1-l}{2}}c_lF_{n-2-l}(X_\de)X_\de\\
&-\sum\limits_{l=1}^{n-3}q^{-\frac{n-1-l}{2}} c_lF_{n-3-l}(X_\de)\\
=&q^{\frac{n}{2}-1}+\sum\limits_{k=1}^{n-2}(\sum\limits_{l=1}^{\text{min}(k,n-1-k)}q^{-1+l})h^{2} +q^{-1}c_{n-3} +\sum\limits_{l=1}^{n-2}q^{-\frac{n-1-l}{2}} c_lF_{n-1-l}(X_\de),
\end{align*}
It suffices to prove that $q^{\frac{n}{2}-1}+\sum\limits_{k=1}^{n-2}(\sum\limits_{l=1}^{\text{min}(k,n-1-k)}q^{-1+l})h^{2} +q^{-1}c_{n-3}=c_{n-1}$.

We have that
$$
c_{n-1}=[\sum\limits_{i=1}^{\frac{n}{2}-1}b_i(q^{-(\frac{n}{2}-i)}+q^{\frac{n}{2}-i})+b_{\frac{n}{2}}]2h^2 +(q^{-(\frac{n}{2}-1)}+q^{-(\frac{n}{2}-3)}+\ldots+q^{\frac{n}{2}-3}+q^{\frac{n}{2}-1}),
$$
\begin{align*}
q^{-1}c_{n-3}=&\Big[\sum\limits_{i=1}^{\frac{n}{2}-2}b_i(q^{-(\frac{n}{2}-i)}+q^{\frac{n}{2}-2-i})+b_{\frac{n}{2}-1}q^{-1}\Big]2h^2 \\ &+(q^{-(\frac{n}{2}-1)}+q^{-(\frac{n}{2}-3)}+\ldots+q^{\frac{n}{2}-5}+q^{\frac{n}{2}-3})
\end{align*}
and $b_k-b_{k-2}=k-1$. Thus
$$
c_{n-1}-q^{-1}c_{n-3}-q^{\frac{n}{2}-1}=\Big[\sum\limits_{k=1}^{\frac{n}{2}}(k-1)q^{\frac{n}{2}-k}\Big]2h^2= \sum\limits_{k=1}^{n-2}(\sum\limits_{l=1}^{\text{min}(k,n-1-k)}q^{-1+l})h^2.
$$

Therefore
\begin{align*}
&q^{\frac{n}{2}-1}+ \sum\limits_{k=1}^{n-2}(\sum\limits_{l=1}^{\text{min}(k,n-1-k)}q^{-1+l})h^2 +q^{-1}c_{n-3} +\sum\limits_{l=1}^{n-2}q^{-\frac{n-1-l}{2}} c_lF_{n-1-l}(X_\de)\\
=&\sum\limits_{l=1}^{n-1}q^{-\frac{n-1-l}{2}} c_lF_{n-1-l}(X_\de)
\end{align*}
and $X_1X_{1+n}=q^{\frac{n}{2}}X_{\frac{n}{2}+1}^{2}+ \sum\limits_{k=1}^{n-1}(\sum\limits_{l=1}^{\text{min}(k,n-k)}q^{-\frac{1}{2}+l})hX_{n+1-k}+\sum\limits_{l=1}^{n-1} q^{-\frac{n-1-l}{2}} c_lF_{n-1-l}(X_\de)$.

When $n$ is odd and $n\geq 5$, we have
$$
X_1X_{n}=q^{\frac{n-1}{2}}X_{\frac{n+1}{2}}^{2}+ \sum\limits_{k=1}^{n-2}(\sum\limits_{l=1}^{\text{min}(k,n-1-k)}q^{-\frac{1}{2}+l})hX_{n-k} +\sum\limits_{l=1}^{n-2} q^{-\frac{n-2-l}{2}} c_lF_{n-2-l}(X_\de)
$$
and
\begin{align*}
&q^{-1}X_1X_{n-1}\\
=&q^{\frac{n-5}{2}}X_{\frac{n-1}{2}}X_{\frac{n+1}{2}} +\sum\limits_{k=1}^{n-3}(\sum\limits_{l=1}^{\text{min}(k,n-2-k)}q^{-\frac{3}{2}+l})hX_{n-1-k} +\sum\limits_{l=1}^{n-3} q^{-\frac{n-1-l}{2}} c_lF_{n-3-l}(X_\de).
\end{align*}
Then
\begin{align*}
&q^{-\frac{1}{2}}X_1X_{n}X_\de\\
=&q^{\frac{n}{2}-1}X_{\frac{n+1}{2}}(q^{-\frac{1}{2}}X_{\frac{n-1}{2}}+q^{\frac{1}{2}}X_{\frac{n+3}{2}}+h) +\sum\limits_{l=1}^{n-2} q^{-\frac{n-1-l}{2}} c_lF_{n-2-l}(X_\de)X_\de\\
&+\sum\limits_{k=1}^{n-2}(\sum\limits_{l=1}^{\text{min}(k,n-1-k)}q^{-1+l})h (q^{-\frac{1}{2}}X_{n-1-k}+q^{\frac{1}{2}}X_{n+1-k}+h) \end{align*}
\begin{align*}
=&q^{\frac{n-3}{2}}X_{\frac{n+1}{2}}X_{\frac{n-1}{2}}+q^{\frac{n-1}{2}}X_{\frac{n+1}{2}}X_{\frac{n+3}{2}} +q^{\frac{n}{2}-1}hX_{\frac{n+1}{2}} +\sum\limits_{k=1}^{n-2}(\sum\limits_{l=1}^{\text{min}(k,n-1-k)}q^{-\frac{3}{2}+l})hX_{n-1-k}\\ &+\sum\limits_{k=1}^{n-2}(\sum\limits_{l=1}^{\text{min}(k,n-1-k)}q^{-\frac{1}{2}+l})hX_{n+1-k} +\sum\limits_{k=1}^{n-2}(\sum\limits_{l=1}^{\text{min}(k,n-1-k)}q^{-1+l})h^2 \\ &+\sum\limits_{l=1}^{n-2}q^{-\frac{n-1-l}{2}}c_lF_{n-2-l}(X_\de)X_\de.
\end{align*}
Note that
\begin{align*}
\left\{
\begin{aligned}
k< n-2-k,&~\text{if~}1\leq k\leq \frac{n-3}{2},\\
k>n-2-k,&~\text{if~}\frac{n-1}{2}\leq k\leq n-3,\\
k\leq n-1-k,&~\text{if~}1\leq k\leq \frac{n-1}{2},\\
k>n-1-k,&~\text{if~}\frac{n+1}{2}\leq k\leq n-3,\\
k< n-k,{\hskip 0.6cm}&~\text{if~}1\leq k\leq \frac{n-1}{2},\\
k>n-k,{\hskip 0.6cm}&~\text{if~}\frac{n+1}{2}\leq k\leq n.
\end{aligned}
\right.
\end{align*}
Hence
\begin{align*}
&\sum\limits_{k=1}^{n-2}(\sum\limits_{l=1}^{\text{min}(k,n-1-k)}q^{-\frac{3}{2}+l})hX_{n-1-k} -\sum\limits_{k=1}^{n-3}(\sum\limits_{l=1}^{\text{min}(k,n-2-k)}q^{-\frac{3}{2}+l})hX_{n-1-k}\\
=&q^{-\frac{1}{2}}hX_1 +\sum\limits_{k=\frac{n+1}{2}}^{n-3}q^{n-\frac{5}{2}-k}hX_{n-1-k} +q^{\frac{n}{2}-2}hX_{\frac{n-1}{2}}\\
=&q^{-\frac{1}{2}}hX_1 +\sum\limits_{k=\frac{n-1}{2}}^{n-3}q^{n-\frac{5}{2}-k}hX_{n-1-k} =\sum\limits_{k=\frac{n+3}{2}}^{n}q^{n-\frac{1}{2}-k}hX_{n+1-k}.
\end{align*}
Note that
\begin{equation*}
q^{\frac{n}{2}-1}hX_{\frac{n+1}{2}}+\sum\limits_{l=1}^{\frac{n-3}{2}}q^{-\frac{1}{2}+l}hX_{\frac{n+1}{2}}= \sum\limits_{l=1}^{\frac{n-1}{2}}q^{-\frac{1}{2}+l}hX_{\frac{n+1}{2}}
\end{equation*}
and
\begin{equation*}
  \sum\limits_{k=\frac{n-1}{2}}^{n-3}q^{n-\frac{5}{2}-k}hX_{n-1-k}=\sum\limits_{k=\frac{n+3}{2}}^{n-1}q^{n-\frac{1}{2}-k}hX_{n-1-k},
\end{equation*}
then we obtain that
\begin{align*}
&q^{\frac{n}{2}-1}hX_{\frac{n+1}{2}} +\sum\limits_{k=1}^{n-2}(\sum\limits_{l=1}^{\text{min}(k,n-1-k)}q^{-\frac{3}{2}+l})hX_{n-1-k} +\sum\limits_{k=1}^{n-2}(\sum\limits_{l=1}^{\text{min}(k,n-1-k)}q^{-\frac{1}{2}+l})hX_{n+1-k}\\ &-\sum\limits_{k=1}^{n-3}(\sum\limits_{l=1}^{\text{min}(k,n-2-k)}q^{-\frac{3}{2}+l})hX_{n-1-k} -q^{-\frac{1}{2}}hX_1\\
=&\sum\limits_{k=1}^{n-1}(\sum\limits_{l=1}^{\text{min}(k,n-k)}q^{-\frac{1}{2}+l})hX_{n+1-k}.
\end{align*}
Since
\begin{align*}
&\sum\limits_{l=1}^{n-2}q^{-\frac{n-1-l}{2}}c_lF_{n-2-l}(X_\de)X_\de -\sum\limits_{l=1}^{n-3} q^{-\frac{n-1-l}{2}} c_lF_{n-3-l}(X_\de) \\
=&\sum\limits_{l=1}^{n-2}q^{-\frac{n-1-l}{2}}c_lF_{n-1-l}(X_\de)+q^{-1}c_{n-3},
\end{align*}
we only need to show that $q^{-1}c_{n-3}+\sum\limits_{k=1}^{n-2}(\sum\limits_{l=1}^{\text{min}(k,n-1-k)}q^{-1+l})h^2 =c_{n-1}$.

Note that $a_k-a_{k-2}=2k-1$ for $k\geq 3$, then
\begin{align*}
&c_{n-1}-q^{-1}c_{n-3}\\
=&\Big[(n-2)+(n-4)q+(n-6)q^2+\ldots+5q^{\frac{n-7}{2}}+3q^{\frac{n-5}{2}}+q^{\frac{n-3}{2}}\Big]h^2 \\
=&\sum\limits_{k=1}^{n-2}(\sum\limits_{l=1}^{\text{min}(k,n-1-k)}q^{-1+l})h^2.
\end{align*}
Therefore
$$
X_1X_{1+n}=q^{\frac{n-1}{2}}X_{\frac{n+1}{2}}X_{\frac{n+3}{2}} +\sum\limits_{k=1}^{n-1}(\sum\limits_{l=1}^{\text{min}(k,n-k)}q^{-\frac{1}{2}+l})hX_{n+1-k} +\sum\limits_{l=1}^{n-1} q^{-\frac{n-1-l}{2}} c_lF_{n-1-l}(X_\de).
$$
The proof is completed.
\end{proof}

\begin{remark}\label{multi-2}
According to \cite[Proposition 4.6]{BCDX} and Lemma \ref{bar},  all cluster variables and $F_{n}(X_\de)\ (n\in \mathbb{Z}_{>0})$ are bar-invariant. Therefore, the cluster multiplication formulas for $F_{n}(X_\delta) F_{m}(X_\delta)$, $F_{n}(X_\delta)X_m$ and $X_{m+n}X_m$ can be obtained by applying the bar-involution to all formulas in Theorem \ref{multi-1}.
\end{remark}

\section{A positive bar-invariant basis of $\A_{q}(2,2)$}

In this section, we will  explicitly construct a positive bar-invariant basis of $\A_{q}(2,2)$.

\begin{definition}
A basis of $\A_{q}(2,2)$ is called a positive $\mathbb{Z}[ {q}^{\pm \frac{1}{2}},h]$-basis if its
structure constants belong to $\mathbb{Z}_{\geq 0}[ {q}^{\pm \frac{1}{2}},h]$.
\end{definition}

Denote
$$\mathcal{B}=\{ q^{-\frac{a_1a_2}{2}}X^{a_1}_mX^{a_2}_{m+1}|m\in\mathbb{Z},(a_1,a_2)\in\mathbb{Z}_{\geq 0}^2 \} \sqcup \{ F_{n}(X_\delta) | n\in \mathbb{Z}_{>0}\}. $$

\begin{lemma}\label{bar-2}
All elements in $\mathcal{B}$ are bar-invariant.
\end{lemma}
\begin{proof}
According to   \cite[Lemma 4.3, Proposition 4.6]{BCDX}, the following equations hold  for any $m\in\mathbb{Z}$:
$$X_mX_{m+1}=qX_{m+1}X_{m},\ \overline{X_m}=X_m.$$
Thus, for any $m\in\mathbb{Z}$ and $(a_1,a_2)\in\mathbb{Z}_{\geq 0}^2$, we have
$$\overline{q^{-\frac{a_1a_2}{2}}X^{a_1}_mX^{a_2}_{m+1}}= q^{\frac{a_1a_2}{2}}X^{a_2}_{m+1}X^{a_1}_m=q^{-\frac{a_1a_2}{2}}X^{a_1}_mX^{a_2}_{m+1}$$
which assert  that all elements in the set
$\{ q^{-\frac{a_1a_2}{2}}X^{a_1}_mX^{a_2}_{m+1}|m\in\mathbb{Z},(a_1,a_2)\in\mathbb{Z}_{\geq 0}^2 \}$ are bar-invariant. Together with Lemma \ref{bar}, we know that any element in $\mathcal{B}$ is bar-invariant.
\end{proof}
 In order to prove that the elements in $\mathcal{B}$ are $\mathbb{Z}[ {q}^{\pm \frac{1}{2}},h]$-independent, we need  the following  definition which gives a partial order $\leq$ on $\mathbb{Z}^2$.
\begin{definition}\label{order}
 Let $(r_1,r_2)$ and  $(s_1,s_2)\in \mathbb{Z}^2$. If $r_i \leq s_i$ for each  $1\leq i \leq 2$, we write  $(r_1,r_2) \leq (s_1,s_2)$. Furthermore, if  $r_i < s_i$ for some $i$, we write $(r_1,r_2)< (s_1,s_2)$.
\end{definition}

\begin{theorem}\label{positive}
The set $\mathcal{B}$ is a positive bar-invariant   $\mathbb{Z}[ {q}^{\pm \frac{1}{2}},h]$-basis of $\A_{q}(2,2)$.
\end{theorem}

\begin{proof}
According to  Theorem~\ref{multi-1} and Remark~\ref{multi-2}, we can deduce
that the generalized quantum cluster algebra  $\A_{q}(2,2)$ is $\mathbb{Z}[ {q}^{\pm \frac{1}{2}},h]$-spanned by  the elements in $\mathcal{B}$.

Note that $X_\delta$ has the minimal non-zero term $X^{(-1,-1)}$ associated to the partial order in Definition \ref{order}, and thus by Theorem \ref{multi-1}, we deduce that the element $F_{n}(X_\delta)$ has the minimal non-zero term $X^{(-n,-n)}$ for each $n\in\mathbb{Z}_{>0}$.
According to  Theorem \ref{multi-2},  we have $X_nX_\delta=q^{\frac{1}{2}}X_{n+1}+q^{-\frac{1}{2}}X_{n-1}+h$. Thus, for each $n\geq 2,$ we obtain that the cluster variable $X_n$ has the minimal non-zero term $a_nX^{-(n-2,n-3)}$ where $a_n\in\mathbb{Z}[ {q}^{\pm \frac{1}{2}}]$, and for each $n\geq -1$,  the cluster variable $X_{-n}$ has the minimal non-zero term $b_nX^{-(n,n+1)}$ where $b_n\in\mathbb{Z}[ {q}^{\pm \frac{1}{2}}]$. Hence, there exists a bijection between the set of all minimal non-zero terms in cluster variables  and  almost positive roots associated to the affine Lie algebra $\hat{\mathfrak{sl}}_2$.  Using the same discussion as  \cite[Proposition 3.1]{SZ}, we have that there exists a bijection between the set of all minimal non-zero terms in the elements in $\mathcal{B}$ and  $\mathbb{Z}^2$, which implies that the elements in  $\mathcal{B}$ are $\mathbb{Z}[ {q}^{\pm \frac{1}{2}}]$-independent.

By using Theorem~\ref{multi-1} and Remark \ref{multi-2} repeatedly,   we can deduce that  the
structure constants of the basis elements in $\mathcal{B}$ belong to $\mathbb{Z}_{\geq 0}[ {q}^{\pm \frac{1}{2}},h]$. Together with Lemma \ref{bar-2}, the proof is completed.
\end{proof}
\begin{remark}\label{special}
If we set $h=0$ and $q=1$, then the set $\mathcal{B}$ is exactly the canonical basis of  the cluster algebra of Kronecker quiver obtained  in \cite{SZ}.
\end{remark}

\begin{definition}
An element in $\A_{q}(2,2)$ is called positive if the coefficients of its Laurent expansion associated to any cluster belong to
$\mathbb{Z}_{\geq 0}[ {q}^{\pm \frac{1}{2}},h]$.
\end{definition}

\begin{remark}\label{positive-element}
According to  Theorem \ref{multi-1} and Remark \ref{multi-2} or using the same arguments as \cite[Corollary 8.3.3]{fanqin1}, it is not difficult to see that every element in
$\mathcal{B}$ is positive. In particular, we obtain that all cluster variables of $\A_{q}(2,2)$ are positive, which is a special case in \cite{rupel}.
\end{remark}

\section*{Acknowledgments}

Liqian Bai was supported by NSF of China (No. 11801445) and the Natural Science Foundation of Shaanxi Province (No. 2020JQ-116), Ming Ding was supported by NSF of China (No. 11771217) and Guangdong Basic and Applied Basic Research
Foundation (2023A1515011739) and Fan Xu was supported by NSF of China (No. 12031007).

\end{document}